\documentclass[11pt]{article}

\usepackage{amsmath,amsthm,amsfonts,latexsym,amssymb, stmaryrd}
\usepackage[all]{xy}
\usepackage{geometry}

\begin{document}

\newtheorem*{theo}{Theorem}
\newtheorem*{pro} {Proposition}
\newtheorem*{cor} {Corollary}
\newtheorem*{lem} {Lemma}
\newtheorem{theorem}{Theorem}[section]
\newtheorem{corollary}[theorem]{Corollary}
\newtheorem{lemma}[theorem]{Lemma}
\newtheorem{proposition}[theorem]{Proposition}
\newtheorem{conjecture}[theorem]{Conjecture}

\theoremstyle{definition}
 \newtheorem{definition}[theorem]{Definition} 
  \newtheorem{example}[theorem]{Example}
   \newtheorem{remark}[theorem]{Remark}
   
\newcommand{\Naturali}{{\mathbb{N}}}
\newcommand{\Reali}{{\mathbb{R}}}
\newcommand{\Complessi}{{\mathbb{C}}}
\newcommand{\Toro}{{\mathbb{T}}}
\newcommand{\Relativi}{{\mathbb{Z}}}
\newcommand{\HH}{\mathfrak H}
\newcommand{\KK}{\mathfrak K}
\newcommand{\LL}{\mathfrak L}
\newcommand{\as}{\ast_{\sigma}}
\newcommand{\tn}{\vert\hspace{-.3mm}\vert\hspace{-.3mm}\vert}
\newcommand{\Mb}{{M^{\rm bim}_0A(\Sigma)}}
\newcommand{\Mbp}{{M^{\rm bim}_0A'(\Sigma)}}
\newcommand{\nd}{{\rm ND}(\alpha)} 
\newcommand{\ndz}{{\rm ND}_0(\alpha)} 
\def\A{{\cal A}}
\def\B{{\cal B}}
\def\D{{\cal D}}
\def\E{{\cal E}}
\def\F{{\cal F}}
\def\H{{\cal H}}
\def\I{{\cal I}}
\def\K{{\cal K}}
\def\L{{\cal L}}
\def\N{{\cal N}}
\def\M{{\cal M}}
\def\gM{{\frak M}}
\def\O{{\cal O}}
\def\P{{\cal P}}
\def\S{{\cal S}}
\def\T{{\cal T}}
\def\U{{\cal U}}
\def\V{{\mathcal V}}
\def\qed{\hfill$\square$}

\title{Negative definite functions for $C^*$-dynamical systems}

\author{Erik B\'edos,
Roberto Conti
\\}
\date{\today}
\maketitle
\markboth{Erik B\'edos, Roberto Conti, }{
}
\renewcommand{\sectionmark}[1]{}
\begin{abstract}
Given an action $\alpha$ of a discrete group on a unital $C^*$-algebra $A$,  
we introduce a natural concept of $\alpha$-negative definiteness for functions from $G$ to $A$, and examine some of the first consequences of such a notion. 
In particular, we prove analogs of theorems due to Delorme-Guichardet and Schoenberg in the classical case where $A$ is trivial.  
We also give a characterization of the Haagerup property for the action $\alpha$ when $G$ is countable.


\end{abstract}

\section{Introduction} \label{Intro}
Given a $C^*$-dynamical system $(A, G, \alpha)$, Anantharaman-Delaroche introduced in \cite{AD1}  the concept of positive definiteness for $A$-valued continuous functions on $G$ relative to the action $\alpha$. 
She also explained how this notion could be used to characterize the amenability of actions of discrete groups on von Neumann algebras and on commutative $C^*$-algebras. 
More recently, it has been shown \cite{DoRu, BeCo6} that any $\alpha$-positive definite function on $G$ taking values in the center of $A$ naturally induces a completely positive map 
both on the reduced and  the full $C^*$-crossed products associated to  a discrete unital system $(A, G, \alpha)$.  

Parallel to the classical notion of positive definiteness for a complex function on a group, 
it has also been very fruitful to consider 
negative definite functions. (By negative definite we always mean the same as what is called conditionally negative definite, or conditionally of negative type, by some authors). Such functions play an important role in characterizing several properties of groups, such as the Haagerup property \cite{CCJJV} and property (T)  \cite{HV, BHV}.
Somewhat surprisingly,  
a study of negative definite functions for $C^*$-dynamical systems so far has been missing in the literature. Our main goal in writing this paper is to start filling this gap by introducing and investigating the first basic concepts. 
 In order to  make this paper easily accessible,
we stick to the case of a unital discrete $C^*$-dynamical system $(A, G, \alpha)$,
 but we do not see see any serious obstruction in extending most of our results to 
 the general case almost  {\it mutatis mutandis}.

We note that  (conditionally) negative definiteness for real functions on locally compact groupoids were introduced by Tu in \cite{Tu99} (see also \cite{Ren10}). As for groups, his definition has a natural generalization to complex functions. 
In the case of the
transformation groupoid associated to an action of a discrete group $G$ on a compact Hausdorff space $\Omega$, it is not difficult to deduce that 
our concept of negative definiteness for a function from $G$ to $C(\Omega)$ (relative to the induced action) is the same as the one obtained after transposing Tu's definition.  
We also mention a very recent paper \cite{MM} of Moslehian where he considers conditionally positive kernels on sets with values in $C^*$-algebras. It should be noted that our definition of $\alpha$-negative definiteness may be formulated by using his terminology (see Remark \ref{mosle}), but that there is otherwise little overlap between our paper and his.
 
Among our main results, we mention 
a Delorme-Guichardet type theorem (cf.~Theorem \ref{cocyclerep}), saying that
a function $\psi$ on $G$ taking values in the positive cone of $A$ and vanishing on the identity of $G$ is  $\alpha$-negative definite if and only if it can be represented
in the form $\psi(g) = \langle c(g),c(g) \rangle$ for a symmetric one-cocycle $c$ relative to an $\alpha$-equivariant action of $G$ on a Hilbert $A$-module.
We also obtain a natural generalization of the classical Schoenberg theorem, which provides a bridge between $\alpha$-positive and $\alpha$-negative definiteness
for center-valued functions on $G$ (cf.~Theorem \ref{schoenberg}). As an application, we obtain a characterization of the Haagerup property for $\alpha$ when $G$ is countable (cf.'Theorem \ref{Haa-sp-prop}). This notion was recently introduced by Dong and Ruan in \cite{DoRu}.  

We hope that the present work will provide useful tools in noncommutative harmonic analysis and potential theory, e.g., in the study of $C^*$-dynamical systems, of
semigroups of completely positive maps, and  of noncommutative Dirichlet forms. 
We  discuss briefly a couple of examples of this sort,
but  we expect that other similar applications will appear soon.
In a different direction, it might be interesting to enlarge our set up
and study  negative definiteness for functions from $G \times A$ into $A$ that are linear in the second variable, 
as we did for positive definiteness  in \cite{BeCo6}. 
We plan to return to this in a subsequent work. 

\section{Preliminaries}\label{Preliminaries} 

Let $A$ be a $C^*$-algebra. We will denote  the center of $A$ by $Z(A)$, the self-adjoint part of $A$ by $A_{\rm sa}$,  the cone of positive elements in $A$ by $A^+$ and the $n\times n$ matrices over $A$ for some natural number $n$ by $M_n(A)$. By a Hilbert $A$-module we will mean a right Hilbert $C^*$-module over $A$, as defined for instance in \cite{La1}.

We record here some lemmas that we will need in the sequel. The first one is proven in \cite{LN} (see Lemma 3.1 therein).
\begin{lemma}\label{schur}
The Schur product of a positive matrix in $M_n(A)$ and a positive element in $M_n(Z(A))$
is still positive in $M_n(A)$.
\end{lemma}

\begin{lemma}\label{transposed0}
Assume $B$ is a commutative $C^*$-algebra, $n\in \mathbb{N}$ and let $[b_{ij}]\in M_n(B)^+$. Then 
$[b^*_{ij}]\in M_n(B)^+$.
\end{lemma}
\begin{proof} We may write $[b_{ij}]=C^*C$ for some $C=[c_{ij}]\in M_n(B)$. Consider $i,j \in \{1,\ldots, n\}$. Then we have
$b_{ij} =\sum_{k=1}^n c_{ki}^* c_{kj}$. Since $B$ is commutative, we get
\[b_{ij}^* = \sum_{k=1}^n c_{kj}^* c_{ki} = \sum_{k=1}^n (c^*_{ki})^* c_{kj}^*\,.\]
Thus, setting $D=[c^*_{ij}] \in M_n(B)$, we get $[b^*_{ij}] = D^*D \in M_n(B)^+$.
\end{proof}
\begin{lemma}\label{transposed}
Let $X$ be a Hilbert $A$-module and assume $x_1,\ldots, x_n\in X$ are such that 
$\langle x_i,x_j \rangle \in Z(A)$, for all $i,j=1,\ldots,n$. Then the transposed matrix 
$\big[\langle x_j,x_i \rangle\big]$ is positive in $M_n(Z(A))$.
\end{lemma}

\begin{proof}
It is well known (cf.~\cite[Lemma 4.2]{La1}) that the matrix $[\langle x_i,x_j \rangle]$ is positive in $M_n(A)$. Since  this matrix lies in $M_n(Z(A))$ by assumption, it follows that $\big[\langle x_i,x_j \rangle\big] \in M_n(Z(A))^+$. Thus, using Lemma \ref{transposed0}, we get\[\big[\langle x_j,x_i \rangle\big] = \big[\,\langle x_i,x_j \rangle^*]  \in M_n(Z(A))^+\,.\]
\vspace{-5ex}\end{proof}

\begin{lemma}\label{schurexp}
Assume $B$ is a commutative $C^*$-algebra.
Let $\Gamma = [\gamma_{ij}] \in M_n(B)$ and let $e^{\circ \Gamma}:=[e^{\gamma_{ij}}] \in M_n(B)$ denote its Schur exponential.
If $\Gamma$ is positive, then $e^{\circ \Gamma}$ is positive too.
\end{lemma}
\begin{proof}
It is well known that the assertion is true when $B = {\mathbb C}$.
Realizing $B$ as the continuous functions on its Gelfand spectrum $\Omega$ and identifying $M_n(B)$ with $C_0(\Omega,M_n(\mathbb{C}))$ in the
natural way, we have $$e^{\circ \Gamma}(\omega) = [e^{\gamma_{ij}}(\omega)] = [e^{\gamma_{ij}(\omega)}] = e^{\circ [\gamma_{ij}(\omega)]} = e^{\circ \Gamma(\omega)} $$
for all $\omega \in \Omega$. Assume now that $\Gamma \in M_n(B)^+$ and let $\omega\in \Omega$. Then we have $\Gamma(\omega) \in M_n(\mathbb{C})^+$, so we get $e^{\circ \Gamma}(\omega) = e^{\circ \Gamma(\omega)} \in M_n(\mathbb{C})^+$. This shows that $e^{\circ \Gamma} \in M_n(B)^+$.
\end{proof}

Let $\alpha:G\to {\rm Aut}(A)$ denote an action of a (discrete) group $G$ on $A$.
Following \cite{AD1, AD2} we will say that a function $\varphi: G \to A\,$ is \emph{$\alpha$-positive definite} 
if for any $n \in {\mathbb N}$ and $g_1,\ldots,g_n \in G$, we have
$$\Big[\alpha_{g_i}\Big(\varphi(g_i^{-1}g_j)\Big)\Big] \geq 0 $$
in $M_n(A)$. In other words,  for any $n \in {\mathbb N}$, $g_1,\ldots,g_n \in G$ and $a_1, \ldots, a_n \in A$, we have
\[\sum_{i,j=1}^n a_i^* \alpha_{g_i}\big(\varphi(g_i^{-1}g_j)\big)a_j \geq 0 \]
in $A$. 
In the scalar case (i.e., $A = {\mathbb C}$), one recovers the classical notion of positive definiteness.

We recall from \cite[Proposition 2.4]{AD1} that 
if $\varphi:G\to A$ is $\alpha$-positive definite, then for every $g\in G$ 
we have
\begin{equation}\label{pdh}
\alpha_g(\varphi(g^{-1})) = \varphi(g)^*\,.
\end{equation}
Moreover, if $e$ denotes the identity of $G$, we have 
\begin{equation}\label{pdh2}\varphi(e) \in A^+\,.\end{equation}
We will also need the following two results.
\begin{lemma}\label{phistar} Assume $\varphi: G\to Z(A)$ is $\alpha$-positive definite. Define $\varphi^*:G\to Z(A)$ by $\varphi^*(g) = \varphi(g)^*$ for each $g\in G$.
Then $\varphi^*$ is $\alpha$-positive definite.
\end{lemma}
\begin{proof} Let $g_1, \ldots, g_n \in G$. Then $\big[\alpha_{g_i}\big(\varphi(g_i^{-1}g_j)\big)\big] \in M_n(Z(A))^+$.  Using Lemma \ref{transposed0}, we get 
\[\big[\alpha_{g_i}\big(\varphi^*(g_i^{-1}g_j)\big)\big] =
\big[\alpha_{g_i}\big(\varphi(g_i^{-1}g_j)^*\big)\big] = \big[\alpha_{g_i}\big(\varphi(g_i^{-1}g_j)\big)^*\big] \in M_n(Z(A))^+\,.\] 
\end{proof}

\begin{lemma}\label{posdefprod} Assume $\varphi_1: G\to A$ and $\varphi_2: G\to Z(A)$ are both $\alpha$-positive definite. Then the pointwise product $\varphi_1\varphi_2$ from  $G$ to $A$ is also $\alpha$-positive definite.
\end{lemma}
\begin{proof} This follows from a straightforward application of Lemma \ref{schur}. 
\end{proof}

\section{Negative definite functions relative to a $C^*$-dynamical system}
Since the concept of (conditionally) negative definiteness for complex functions on groups is useful in many contexts (see e.g.~\cite{BCR, CCJJV}),
it is natural to 
investigate a notion of negative definiteness relative to 
$C^*$-dynamical systems. 
Throughout this paper, we let 
$\alpha:G\to {\rm Aut}(A)$ 
denote an action of a (discrete) group $G$ on a  unital $C^*$-algebra $A$ and let  
$A^\alpha=\{a\in A\mid \alpha_g(a) = a \text{ for all } g\in G\}$ denote the fixed-point algebra of $A$ under $\alpha$. 
The identity element of $G$ will be denoted by $e$ and the unit of $A$ will be denoted by $1_A$.

The following definition is the natural generalization of the classical notion.

\begin{definition}\label{negdef}
We will say that a function $\psi: G \to A$ is \emph{$\alpha$-negative definite} if 
\[\alpha_g(\psi(g^{-1})) = \psi(g)^*\] for all $g \in G$ and, 
for any $n \in {\mathbb N}$, $g_1,\ldots,g_n \in G$ and $b_1,\ldots,b_n \in A$ with
$\sum_{i=1}^n b_i = 0$, we have
$$\sum_{i,j=1}^n b_i^* \alpha_{g_i}\big(\psi(g_i^{-1}g_j)\big)b_j \leq 0  $$
in $A$.
We will  say that an $\alpha$-negative definite function $\psi$ is \emph{normalized} when $\psi(e) = 0$. 
 Moreover, we will let $\nd$ denote the set of all $\alpha$-negative definite functions and set 
 $\ndz=\{ \psi \in \nd\mid \psi(e)=0\}$.
\end{definition}
Clearly, $\nd$ contains every constant function from $G$ to $A$ of the form $g\to t\, 1_A$ for some $t\in \mathbb{R}$.  Also, it follows immediately that $\nd$ is a cone (that is, the sum of $\alpha$-negative definite functions as well as any positive multiple of an $\alpha$-negative definite function are again $\alpha$-negative definite) and that 
$\ndz$ is a subcone of $\nd$. 
Moreover, we have:  
\begin{lemma}\label{pointwise}
The cones $\nd$ and $\ndz$ 
are closed w.r.t.~the pointwise norm-topology.
\end{lemma}
\begin{proof}
As $\ndz$ is closed in $\nd$ with respect to the pointwise norm-topology, it suffices to prove the assertion for $\nd$.
Assume that $\{\psi_\beta\}$ is a net in $\nd$ converging to some $\psi: G \to A$ w.r.t.~the pointwise norm topology.
Then for every $g \in G$ we have \[\alpha_g(\psi(g^{-1})) = \lim_\beta \alpha_g(\psi_\beta(g^{-1})) = \lim_\beta \psi_\beta(g)^* = \psi(g)^*\,.\]
Moreover, let $g_1,\ldots,g_n \in G$ and let $b_1, \ldots, b_n \in A$ satisfy $\sum_{i=1}^n b_i = 0$. Then for every $\beta$ we have
\[\sum_{i,j=1}^n b_i^* \alpha_{g_i}\big(\psi_\beta(g_i^{-1}g_j)\big)b_j \leq 0\,,\]  so we get
\[\sum_{i,j=1}^n b_i^* \alpha_{g_i}\big(\psi(g_i^{-1}g_j)\big)b_j 
= \lim_\beta\Big(\sum_{i,j=1}^n b_i^* \alpha_{g_i}\big(\psi_\beta(g_i^{-1}g_j)\big)b_j \Big) \leq 0 \]
since $A^+$ is norm-closed in $A$. Thus $\psi \in \nd$. 
\end{proof}

\begin{remark}\label{basicnegdef}
Let $\psi \in \nd$. 
Then we have $\psi(e)^* = \alpha_e\big(\psi(e^{-1})\big) = \psi(e)\,,$ so $\psi(e) \in A_{\rm sa}$.  
Moreover, taking $n=2$, $g_1 = e$, $g_2 = g$, $b_1 = 1_A = -b_2$ in Definition \ref{negdef}, we get 
\[\psi(e) - \psi(g) - \alpha_g(\psi(g^{-1})) + \alpha_g(\psi(e))
= \psi(e) + \alpha_g(\psi(e)) - 2\, {\rm Re}\big( \psi(g)\big) \leq 0\,,\]
hence  \[{\rm Re}\big( \psi(g)\big) \geq \frac{1}{2}\big(\psi(e) + \alpha_g(\psi(e))\big)\] for all $g \in G$.
In particular,  if $\psi(e) \geq 0$, then
$\,{\rm Re} \big(\psi(g)\big) \in A^+$ for all $g \in G$.
\end{remark}

\begin{remark}
Let $\psi:G\to A$ and define $\psi_0 : G \to A$  by \[\psi_0(g) = \psi(g) - \psi(e)\,\] for every $g\in G$,  so $\psi_0(e) = 0$. Assume that $\psi(e) \in A_{\rm sa} \cap A^\alpha$. We leave to the reader to
verify that $\psi \in \nd $ if and and only if $\psi_0 \in \ndz$.
\end{remark}

\begin{remark}\label{oneminus}
Let $\varphi: G \to A$ be $\alpha$-positive definite. In particular, $\varphi(e) \in A_{\rm sa}$, cf.~(\ref{pdh2}). Assume that we also have $\varphi(e) \in A^\alpha$.  
Then the function $\psi: G\to A$ defined by \[\psi(g)= \varphi(e) - \varphi(g)\] 
belongs to $\ndz$.

Indeed, consider $g\in G$. Then, using (\ref{pdh}), we have
\[\alpha_g(\psi(g^{-1}))=\alpha_g\big(\varphi(e) - \varphi(g^{-1})\big) = \varphi(e) - \alpha_g(\varphi(g^{-1}))
= \varphi(e)^* - \varphi(g)^*=\psi(g)^*\,.\] 
Moreover, for any $g_1,\ldots, g_n \in G$ and $b_1,\ldots,b_n \in A$ with $\sum_{i=1}^n b_i = 0$, we have 
$$
\sum_{i,j=1}^n b_i^* \alpha_{g_i}\big(\psi(g_i^{-1}g_j)\big)b_j=\sum_{i,j=1}^n b_i^* \alpha_{g_i} \big(\varphi(e) - \varphi(g_i^{-1}g_j)\big) b_j
= - \sum_{i,j=1}^n b_i^*  \alpha_{g_i} \big(\varphi(g_i^{-1}g_j)\big) b_j \leq 0 \,,
$$ 
as desired.
\end{remark}

\begin{proposition}\label{gamma}
Let $\psi:G\to A$ and define $\gamma:G\times G \to A$ by
\begin{equation}\label{kernel}
\gamma(g,h) = 
\psi(g)^*  + \psi(h) -\psi(e) - \alpha_g(\psi(g^{-1}h)) 
\end{equation}
for all $g,h \in G$.
Then the following two assertions are equivalent:
\begin{itemize}
\item[$($i$)$] $\psi \in \nd;$
\item[$($ii$)$]
For every $n \in \mathbb{N}$ and $g_1,\ldots,g_n \in G$, the matrix $[\gamma(g_i,g_j)]$ is positive in $M_n(A)$. 
\end{itemize}
\end{proposition}

\begin{proof}
Suppose that $\psi \in ND(\alpha)$ and let $g_1,\ldots,g_n \in G$, 
 $b_1,\ldots,b_n \in A$. Set $b_0 = - \sum_{i=1}^n b_i$ and $g_0 = e$.
Then $\sum_{i=0}^n b_i = 0$, so we get 
$$\sum_{i,j=0}^n b_i^* \alpha_{g_i}\big(\psi(g_i^{-1}g_j)\big) b_j \leq 0 \,. $$
This gives that
$$b_0^*\,\psi(e)\,b_0\,-\Big(\sum_{i=1}^n b_i^*\Big) \sum_{j=1}^n \psi(g_j) b_j\, - \sum_{i=1}^n b_i^* \alpha_{g_i}\big(\psi(g_i^{-1})\big) \Big(\sum_{j=1}^n b_j\Big)
+ \sum_{i,j=1}^n b_i^* \alpha_{g_i}\big(\psi(g_i^{-1}g_j)\big)b_j \leq 0 \,. $$
As $b_0^*\,\psi(e)\,b_0= \sum_{i,j=1}^n b_i^*\,\psi(e)\,b_j$ and $\alpha_{g_i}\big(\psi(g_i^{-1})\big) = \psi(g_i)^*$ for every $i$, this gives that 
$$\sum_{i,j=1}^n b_i^*\gamma(g_i, g_j)b_j=-\sum_{i,j=1}^n b_i^* \Big( \psi(e) -\psi(g_j) - \psi(g_i)^* + \alpha_{g_i}\big(\psi(g_i^{-1}g_j)\big) \Big) b_j \, \geq 0 \,. $$
Thus we have shown that $(ii)$ holds. 

Conversely, suppose that $(ii)$ is true. Note first that $\gamma(e,e) = \psi(e)^* -\psi(e) $. Since $\gamma(e, e)\in A^+$, we get that $\gamma(e,e) = \gamma(e,e)^* =  \psi(e)-\psi(e)^*  = -\gamma(e,e)$, so $\gamma(e,e) = 0$, hence $\psi(e)^*=\psi(e)$. Let now $g\in G$. 
Note that $\gamma(g,e) = 
\psi(g)^*-\alpha_g(\psi(g^{-1}))$,
while $\gamma(e,g) = 
 \psi(e)^*+\psi(g)-\psi(e) -\psi(g)=\psi(e)^*-\psi(e)= 0$. Since 
\[\left[\begin{array}{cc} \gamma(e,e) &\gamma(e,g)\\ \gamma(g,e) &\gamma(g,g)\end{array}\right] \in M_2(A)^+\] 
we have $\gamma(e,g)^*= \gamma(g,e)$. Thus, we get 
$\psi(g)^*-\alpha_g(\psi(g^{-1}))
=\gamma(g,e) = \gamma(e,g)^*= 0$, that is, $\alpha_g(\psi(g^{-1})) = \psi(g)^*$. 

Next, consider 
$g_1,\ldots,g_n \in G$ and $b_1,\ldots,b_n \in A$ with
$\sum_{i=1}^n b_i = 0$. Then 
\begin{align*} 
\sum_{i,j=1}^n & b_i^* \alpha_{g_i}\big(\psi(g_i^{-1}g_j)\big)b_j\\
& = \sum_{i,j=1}^n b_i^* \alpha_{g_i}\big(\psi(g_i^{-1}g_j)\big)b_j
-\Big(\sum_{i=1}^n b_i\Big)^* \sum_{j=1}^n \psi(g_j) b_j 
- \sum_{i=1}^n b_i^* \alpha_{g_i}(\psi(g_i^{-1})) \Big(\sum_{j=1}^n b_j\Big)\\
&=-\Big(\sum_{i,j=1}^n b_i^*\gamma(g_i, g_j)b_j\Big)\leq 0\,,
\quad \text{since} \big[\gamma(g_i,g_j)\big] \in M_n(A)^+\,. 
\end{align*} Thus $\psi \in ND(\alpha)$, 
as desired.
\end{proof}

\begin{remark}\label{mosle} In a very recent work \cite{MM}, Moslehian studies positive and conditionally positive kernels on sets with values in $C^*$-algebras. One easily sees that a function $\psi : G\to A$
 is $\alpha$-negative definite in our sense if and only if the kernel from $G\times G$ into $A$ given by $ K(g,h) = -\,\alpha_g(\psi(g^{-1}h))$ is Hermitian and conditionally positive as defined in \cite{MM}. Proposition \ref{gamma}, which says that $\psi$ is $\alpha$-negative definite if and only if the kernel $\gamma$ is positive, may then be deduced from \cite[Theorem 2.4]{MM}. We have included a self-contained proof of Proposition \ref{gamma} for the ease of the reader. There is otherwise very little overlap between our paper and \cite{MM}.   

\end{remark}

\begin{remark} \label{scalarnegdef}
Let $f: G \to {\mathbb C}$ and consider $\widetilde{f}: G \to A$ defined by $\widetilde{f}(g) = f(g)\,1_A$. If
the function $\widetilde{f}$ is $\alpha$-negative definite, then it is immediate that $f$ is negative definite.  
Conversely,  if $f$ is negative definite, then 
the kernel 
\[F(g,h):=\overline{f(g)} + f(h) - f(e) + f(g^{-1}h) \in \mathbb{C}\] on $G\times G$ is positive (cf.~Proposition \ref{gamma}). But this implies that the kernel $\widetilde{F}(g, h) := F(g,h)\,1_A \in A$ is positive, and Proposition \ref{gamma} gives now that  $\widetilde{f}$ is $\alpha$-negative definite.
\end{remark}

We will let 
$\alpha'$ denote the action of $G$ on  $Z(A)$ obtained by restricting each $\alpha_g$ to $Z(A)$. 
If $\psi \in {\rm ND}(\alpha)$ is  $Z(A)$-valued, then obviously
$\psi \in {\rm ND}(\alpha')$.
Conversely, assume that $\psi \in {\rm ND}(\alpha')$.
 Then $\psi$ is $Z(A)$-valued, and 
$\alpha_g(\psi(g^{-1})) = \alpha'_g(\psi(g^{-1})) = \psi(g)^*$ for every $g \in G$. Moreover, let $\gamma$ be defined as in (\ref{kernel}). Of course, we also have
\[\gamma(g,h) = 
  \psi(g)^* + \psi(h) 
 -\psi(e) - \alpha'_g(\psi(g^{-1}h))
 \]
for all $g, h \in G$. Now, consider $g_1,\ldots,g_n \in G$. Using Proposition \ref{gamma} (with $\alpha'$ instead of $\alpha$), we get that  $\Gamma:=\big[\gamma(g_i,g_j)\big] \in M_n(Z(A))^+$, and this implies that $\Gamma \in M_n(A)^+$. Proposition \ref{gamma} (now with $\alpha$) gives that $\psi \in \nd$.  
 This means that we have:
\begin{proposition}\label{normalized}
One has \, 
${\rm ND}(\alpha')  = \big\{\psi \in \nd \ | \ \psi 
\ \text{is 
$Z(A)$-valued}\big\}.$
\end{proposition}

\smallskip  Assume that there exists a conditional expectation $E : A \to Z(A)$ satisfying $\alpha'\circ E = E\circ \alpha$. Then the map $\psi \to E\circ \psi$ gives a surjection from $\nd$ onto ${\rm ND}(\alpha')$. This follows readily from the fact that a conditional expectation is completely positive (cf.~\cite{BrOz}). Similarly, if there exists a state $\omega$ on $A$ which is $\alpha$-invariant, then  the map $\psi \to \omega \circ \psi$ gives a surjection from $\nd$ onto ${\rm ND}(G)$, the complex negative definite functions on $G$.
\medskip

We also record the following:

\begin{proposition}\label{commut}
If $\psi \in {\rm ND}(\alpha')$, then  ${\rm Re}\, \psi \in {\rm ND}(\alpha')$. If, in addition, $\psi$ is normalized, then ${\rm Re}\, \psi$ is $Z(A)^+$-valued.

\end{proposition}
\begin{proof} 
We may assume that $A$ is commutative and $\alpha'=\alpha$. So consider $\psi \in \nd$. We have to show that ${\rm Re}\, \psi \in \nd$.

One readily verifies that for each $g\in G$ we have $\alpha_g\big(({\rm Re}\, \psi)(g^{-1})\big) =  
({\rm Re}\, \psi)(g)^*$. Next, consider $g_1,\ldots,g_n \in G$ and $b_1,\ldots,b_n \in A$ with
$\sum_{i=1}^n b_i = 0$.  Then
\[\sum_{i,j=1}^n b_i^*\, \alpha_{g_i}\big(({\rm Re}\, \psi)(g_i^{-1}g_j)\big)\,b_j 
=\frac{1}{2} \sum_{i,j=1}^n b_i^* \alpha_{g_i} \big( \psi(g_i^{-1}g_j)\big)b_j + \frac{1}{2}\sum_{i,j=1}^n b_i^* \alpha_{g_i}  \big(\psi(g_i^{-1}g_j)^*\big)b_j
\]
Since the first term on the right hand-side of this equality is negative, it suffices to show that the second term is also negative. Now, using that $A$ is commutative, we have
\[\sum_{i,j=1}^n b_i^* \alpha_{g_i}  \big(\psi(g_i^{-1}g_j)^*\big)b_j = 
\sum_{i,j=1}^n \big(b_j^* \alpha_{g_i}  \big(\psi(g_i^{-1}g_j)\big)b_i\big)^*
=\Big(\sum_{i,j=1}^n (b_i^*)^* \alpha_{g_i}  \big(\psi(g_i^{-1}g_j)\big)b_j^*\Big)^*
\]
which is negative since $\sum_{i=1}^n b_i^* = 0$ and $\psi \in \nd$. Thus, ${\rm Re}\, \psi \in \nd$. 

Finally, if $\psi \in \ndz$, then Remark \ref{basicnegdef} gives that ${\rm Re}\, \psi$ is $A^+$-valued. 
\end{proof}

In general, we do not know whether  ${\rm Re}\, \psi$ belongs to $\nd$ whenever $\psi \in \nd$.
\begin{remark}
Let $\psi\in \nd$ and assume $\psi$ takes its values in $A^+$ (or in $Z(A)^+$). When $A = \mathbb{C}$, it is known that $\psi^{1/2}$ (or, more generally, $\psi^\beta$ with $0 < \beta < 1$)  is still 
$\alpha$-negative definite, see for example \cite[Corollary 2.10]{BCR}. 
One might wonder whether this holds in general.
The first condition for $\alpha$-negative definitess of  $\psi^{1/2}$ is satisfied since for every $g\in G$ we have $\alpha_g(\psi^{1/2}(g^{-1})) =\alpha_g((\psi(g^{-1}))^{1/2})= \big(\alpha_g(\psi(g^{-1}))\big)^{1/2} =\psi(g)^{1/2} = \psi^{1/2}(g)$. 
However, it is not obvious how to proceed to handle the second condition.
\end{remark}

It is now time to introduce a natural class of normalized $\alpha$-negative definite functions related to $\alpha$-equivariant actions of $G$ on Hilbert $A$-modules and one-cocycles for such actions,
much in the same way as normalized complex negative definite functions on $G$ are related to unitary representations of $G$ on Hilbert spaces and their associated one-cocycles. We recall from \cite{AD1} (see also \cite{Com}) that 
an {\it $\alpha$-equivariant action} $u$ of $G$ on a Hilbert $A$-module $X$ is a  homomorphism $u: g\mapsto u_g$ from $G$ into the group $\mathcal{I}(X)$ of bijective $\mathbb{C}$-linear isometries from $X$ into itself, satisfying:
\begin{itemize}
\item[(i)] $\alpha_g\big(\langle x, y\rangle\big) =  \langle u_g x, u_g y\rangle$, \quad\text{and}
\item[(ii)] $u_g (x\cdot a) = (u_g x)\cdot \alpha_g(a)$,
\end{itemize}
for all $g\in G, \,x,y\in X$, and $a \in A$.

\begin{definition} We will say that  $x \in X$ is $u$-{\it symmetric} if 
$\langle x, u_g x \rangle \in A_{\rm sa}  $ 
for all $g \in G$,
and that it is $u$-{\it central} if 
$\langle  x, u_g x \rangle \in Z(A) $ 
for all $g\in G$.
\end{definition}
It follows easily by using property (i) of $u$ that  $x\in X$ is $u$-symmetric (resp.~$u$-central) if and only if $\langle u_g x, u_h x \rangle$ belongs to $ A_{\rm sa}  $ (resp.~$Z(A)$) for all $g, h \in G$. 

\begin{example}
Let $x \in X$ be $u$-symmetric (resp.~$u$-central) and let $\psi: G \to A^+$ (resp.~$Z(A)^+$) be defined by
$$\psi(g) = \langle u_g x - x, u_g x - x \rangle\,.$$ 
Then $\psi \in \ndz$. 
Indeed, it is clear that $\psi(e) = 0$, and for every $g\in G$  we have
$$\alpha_g\big(\psi(g^{-1}\big)= \alpha_g \big(\langle u_{g^{-1}} x - x, u_{g^{-1}} x - x \rangle\big) 
= \langle x - u_g x, x - u_g x \rangle = \psi(g) = \psi(g)^*\,.$$
 Moreover,
for any $g_1,\ldots, g_n \in G$ and $b_1,\ldots,b_n \in A$ with $\sum_{i=1}^n b_i = 0$, we have
\begin{align*}
\sum_{i,j=1}^n  b_i^* \alpha_{g_i}\big(\psi(g_i^{-1}g_j) \big) b_j
&=\sum_{i,j=1}^n  b_i^* \alpha_{g_i}\big(\langle u_{g_i^{-1}g_j}x - x,u_{g_i^{-1}g_j}x - x \rangle \big) b_j \\
&=  \sum_{i,j=1}^n b_i^* \langle u_{g_j}x - u_{g_i}x,u_{g_j}x - u_{g_i}x \rangle b_j \\
& = \Big(\sum_{i=1}^n b_i^* \Big) \Big(\sum_{j=1}^n \langle u_{g_j}x,u_{g_j}x \rangle b_j \Big) 
- \Big(\sum_{i,j=1}^n b_i^* \langle u_{g_j}x,u_{g_i}x \rangle b_j \Big) \\
& \quad  - \Big(\sum_{i,j=1}^n b_i^* \langle u_{g_i}x,u_{g_j}x \rangle b_j \Big)
+ \Big(\sum_{i=1}^n b_i^* \langle u_{g_i}x,u_{g_i}x \rangle \Big) \Big(\sum_{j=1}^n b_j \Big) \\
& =  - \sum_{i,j=1}^n b_i^* \langle u_{g_j}x,u_{g_i}x \rangle b_j 
- \sum_{i,j=1}^n b_i^* \langle u_{g_i}x,u_{g_j}x \rangle b_j \,. 
\end{align*}
The last expression can be seen to be negative without too much difficulty. As we will show this in the proof of Proposition \ref{cocyclenegdef} in a more general case, we skip the argument. 
\end{example}
We notice that a $u$-symmetric (resp.~$u$-central) vector $x$
gives rise to a symmetric (resp. central) one-cocycle $c : G\to X$ (w.r.t.~$u$), as defined below,
by setting 
 $c(g) = u_g x - x$ for each $g \in G$. 
\begin{definition}
 A map $c: G \to X$ will be called a \emph{one-cocycle} $($w.r.t.~$u$$)$ if it satisfies that
$$c(gh) = c(g) + u_g(c(h))\,,\quad \text{for all}\  g,h \in G \,. $$
Moreover, such a one-cocycle $c$ will be called
\begin{itemize}
\item[(i)]  \emph{symmetric} if   
$\langle c(g),c(h) \rangle \in A_{\rm sa}$ for all $g,h \in G$, or, equivalently, if\\
$$\langle c(g),c(h) \rangle = \langle c(h),c(g) \rangle \quad \text{ for all } \ g,h \in G\,;$$
\item[(ii)] \emph{central} if 
$\langle c(g),c(h) \rangle \in Z(A)$ for all $g,h \in G$.
\end{itemize}
\end{definition}
One-cocycles (wr.r.t.~$u$) of the form $c(g) = u_g x - x $ should be thought of as coboundaries.

\begin{remark}
Assume that $c:G\to X$ is a one-cocycle (w.r.t.~$u$).
For each $g \in G$ one may define a bijective affine map $a_g: X \to X$ by $$a_g x = u_g x + c(g)\,.$$ Then 
each $a_g$ is isometric in the sense that
$\|a_g x - a_g y \| = \|x - y\|$ for all $x,y \in X$, and one easily checks that  $a_{gh} = a_g \,a_h$ for all $g,h \in G$.
Hence $g \mapsto a_g$ is a homomorphism from $G$ into the group of affine isometric bijections 
from $X$ into itself.
\end{remark}
\begin{proposition}\label{cocyclenegdef}
 Let $c: G \to X$ be a one-cocycle $($w.r.t.~$u$$)$
and suppose that $c$ is symmetric $($resp.~central$)$.
Let $\psi:G\to A^+$ $($resp. $Z(A)^+$$)$ be defined  for each $g\in G$ by
$$\psi(g) = \langle c(g),c(g) \rangle\,.$$ 
Then $\psi \in \ndz$.
\end{proposition}

\begin{proof}
One has $c(g) = c(ge) = c(g) + u_g(c(e))$, thus $c(e) = 0$ and it follows at once that $\psi(e) = 0$.
Now, $0 = c(e) = c(g^{-1}g) = c(g^{-1}) + u_{g^{-1}}(c(g))$, that is 
$c(g^{-1}) = - u_{g^{-1}}(c(g))$ and therefore
$$\alpha_g(\psi(g^{-1})) =  \alpha_g \big(\langle c(g^{-1}),c(g^{-1}) \rangle\big) = \langle u_g(c(g^{-1})),u_g(c(g^{-1})) \rangle
= \langle c(g), c(g) \rangle = \psi(g)=\psi(g)^*$$ for every $g \in G$. 
Now, for $g,h \in G$, we have
\begin{align*}
\langle c(g^{-1}h), c(g^{-1}h) \rangle 
& = \big\langle c(g^{-1}) + u_{g^{-1}}(c(h)), c(g^{-1}) + u_{g^{-1}}(c(h)) \big\rangle \\
& =\big\langle u_{g^{-1}}(c(h) - c(g)), u_{g^{-1}}(c(h)-c(g)) \big\rangle \,,
\end{align*}
so we get
$$\alpha_g(\psi(g^{-1}h)) = \alpha_g \big(\langle c(g^{-1}h), c(g^{-1}h) \rangle \big) = \langle c(h) - c(g),c(h)-c(g) \rangle \,.$$
Hence, 
for any given $g_1,\ldots, g_n \in G$ and $b_1,\ldots,b_n \in A$ with $\sum_{i=1}^n b_i = 0$, we have
$$
\sum_{i,j = 1}^n b_i^* \alpha_{g_i}\big(\psi(g_i^{-1}g_j) \big) b_j
= \sum_{i,j = 1}^n b_i^* \big\langle c(g_j) - c(g_i), c(g_j) - c(g_i) \big\rangle b_j\, .
$$
The last sum above is negative if $c$ is symmetric or central. 
Indeed, if $c$ is symmetric, then 
$$\sum_{i,j = 1}^n b_i^* \big\langle c(g_j) - c(g_i), c(g_j) - c(g_i) \big\rangle b_j 
= - 2 \sum_{i,j=1}^n b^*_i \langle c(g_i), c(g_j) \rangle b_j \,, $$ 
which is negative since the matrix $[\langle c(g_i), c(g_j) \rangle]$ is positive (cf.~\cite{La1}).
If $c$ is central, then
$$\sum_{i,j = 1}^n b_i^* \big\langle c(g_j) - c(g_i), c(g_j) - c(g_i) \big\rangle b_j 
= - \sum_{i,j=1}^n b^*_i \langle c(g_i), c(g_j) \rangle b_j - \sum_{i,j=1}^n b^*_i  b_j \langle c(g_j), c(g_i) \rangle$$ 
which is seen to be negative by using Lemma  \ref{schur} and Lemma \ref{transposed}.
\end{proof}

A well known result of Delorme and Guichardet  \cite{Del, Gui} says that any  normalized negative definite function $f: G \to {\mathbb R}^+$ 
can be written in the form $f(s) = \|c(s)\|^2$ for a suitable unitary representation $\pi$ of $G$ on a Hilbert space $H$
and a one-cocycle $c$ for $\pi$, i.e., a map $c: G \to H$ satisfying $c(gh) = c(g) + \pi_g\big(c(h)\big)$ for all $g,h \in G$. In our context, as a converse to Proposition \ref{cocyclenegdef}, we have the following analogous result:

\begin{theorem}\label{cocyclerep}
Let $\psi: G \to A^+$ be a normalized $\alpha$-negative definite function.
Then there exists  a Hilbert $A$-module $X$, an $\alpha$-equivariant action $u$ of $G$ on $X$ and a symmetric one-cocycle 
$c: G \to X$ $($w.r.t. $u$$)$ such that $$\psi(g) = \langle c(g),c(g) \rangle$$ for all $g \in G$. 
Moreover, the $A$-submodule of $X$ generated by the $c(g)$'s is dense in $X$. 
Finally, if $\psi$ takes values in $Z(A)^+$, then $c$ is also central.
\end{theorem}

\begin{proof}
For every $(g,h) \in G\times G$, we set
$$\gamma(g,h) = \frac{1}{2} \Big(\psi(g) + \psi(h) - \alpha_g(\psi(g^{-1}h))\Big) \in A_{\rm sa}\,.$$ 
Note that since $\psi(e) = 0$, this agrees with the expression for $\gamma(g,h)$ given in (\ref{kernel}), except for the normalization factor $1/2$. Since $\psi$ is $\alpha$-negative definite, we have that 
$\alpha_g(\psi(g^{-1}h)) = \alpha_h\big(\alpha_{h^{-1}g}(\psi((h^{-1}g)^{-1}))\big)= \alpha_h(\psi(h^{-1}g))$
and it readily follows from this equality that $\gamma(g,h) = \gamma(h,g)$ for all $g,h \in G$.
Moreover, according to Proposition \ref{gamma}, we have 
$$\sum_{i,j=1}^n b_i^* \gamma(g_i,g_j) b_j \geq 0$$
for all $g_1,\ldots,g_n \in G$ and $b_1,\ldots, b_n \in A$.

Let now $X_0 := C_c(G,A)$ denote the space of all $A$-valued, finitely supported  functions on $G$. We can then define a right action of $A$ on $X_0$
by $(f \cdot a)(g) = f(g)a$ for every $f \in X_0$ and every $a \in A$,
and 
an $A$-valued semi-inner product on $X_0$ by
$$\langle f_1, f_2 \rangle_0 := \sum_{g,h \in G} f_1(g)^* \gamma(g,h) f_2(h) \,. $$
As usual, setting $N = \{f \in X_0 \ | \ \langle f, f \rangle_0 = 0\}$ and defining 
$$\langle f_1 + N, f_2 +N \rangle := \langle  f_1, f_2 \rangle_0 \,, $$
$X_0 / N$ becomes an inner product $A$-module. We let $X$ be its Hilbert $A$-module completion and identify $X_0/N$ with its canonical image in $X$.

Next, we define $c: G \to X$ by
$$c(g) := (\delta_g \odot 1_A) + N \quad \text{for each }  g \in G \,, $$
where $\delta_g \odot 1_A$ denotes the function in $X_0$ which takes the value $1_A$ at $g$ and is zero otherwise.   
Then we clearly have that  $X_0/N = {\rm Span}\big\{c(g) \cdot a \ | \ g \in G, a \in A \big\}$, 
so the $A$-submodule of $X$ generated by the $c(g)$'s is dense in $X$.
We also note that 
\begin{equation}\label{cgamma}
\langle c(g),c(h) \rangle = \langle \delta_g \odot 1_A, \delta_h \odot 1_A \rangle_0 =  \gamma(g,h)
\end{equation}
 for all $g,h \in G$, which immediately yields that $c$ is symmetric. This also gives that 
 $\langle c(e),c(e)\rangle = \gamma(e,e) = \psi(e)/2 =0$, so that $c(e) = 0$.
Moreover, using (\ref{cgamma}), we get that for all $g,h,h' \in G$, 
\begin{align*}
\big\langle c(gh) - c(g) &, c(gh') - c(g) \big\rangle =  \gamma(gh,gh') -  \gamma(gh,g) - \gamma(g,gh') + \gamma(g,g) \\
& = \frac{1}{2}\Big[ \psi(gh) + \psi(gh') - \alpha_{gh}(\psi((gh)^{-1}gh')) - \psi(gh) - \psi(g)  
+ \alpha_{gh}(\psi((gh)^{-1}g)) \\
& \quad \quad - \psi(g) - \psi(gh') + \alpha_g(\psi(g^{-1}gh')) + 2\, \psi(g)\Big] \\
& = \frac{1}{2}\, \alpha_g\Big( - \alpha_h(\psi(h^{-1}h')) + \alpha_h(\psi(h^{-1})) + \psi(h')\Big) = \alpha_g\big(\gamma(h,h') \big)\\
& = \alpha_g\Big(\big\langle c(h),c(h') \big\rangle \Big) \,. 
\end{align*}
Consider now $g,g_1,\ldots,g_n,g'_1,\ldots,g'_m \in G$, $a_1,\ldots,a_n, a'_1,\ldots,a'_m \in A$, and
$$
F = \sum_{i=1}^n c(g_i) \cdot a_i\,, \quad F' = \sum_{j=1}^m c(g'_j) \cdot a'_j \,,$$ 
$$U = \sum_{i=1}^n \big(c(gg_i) - c(g)\big)\cdot \alpha_g(a_i)\,,  \quad U' = \sum_{j=1}^m \big(c(gg'_j) - c(g)\big)\cdot \alpha_g(a'_j) $$
in $X_0/N$.
Then, using our previous observation, we get
\begin{align*}
\langle U , U' \rangle & = \sum_{i,j} \big\langle (c(gg_i)-c(g)) \cdot \alpha_g(a_i), (c(gg'_j)-c(g))\cdot \alpha_g(a'_j) \big\rangle \\
& = \sum_{i,j} \alpha_g(a_i)^* \big\langle c(gg_i)-c(g), c(gg'_j)-c(g)\big\rangle \alpha_g(a'_j) \\
& = \sum_{i,j} \alpha_g(a_i)^* \, \alpha_g \big(\langle c(g_i), c(g'_j)\rangle\big) \, \alpha_g(a'_j)  \\
& = \alpha_g\Big(\sum_{i,j} \langle c(g_i) \cdot a_i, c(g'_j) \cdot a'_j \rangle \Big) \\
& = \alpha_g(\langle F, F' \rangle) \ . 
 \end{align*}
Hence, $\|U\| = \| \langle U,U \rangle \|^{1/2}_A = \| \alpha_g(\langle F,F \rangle) \|^{1/2}_A = \|\langle F,F \rangle \|^{1/2}_A = \|F\|$.

As $X_0/N = {\rm Span}\big\{c(g) \cdot a \ | \ g \in G, a \in A \big\}$, we see that, for each $g \in G$, the map
$u_g: X_0/N \to X_0/N$ given by 
$$u_g \Big(\sum_{i=1}^n c(g_i) \cdot a_i\Big) = \sum_{i=1}^n \big(c(gg_i) - c(g) \big) \cdot \alpha_g(a_i)$$
is well-defined, isometric and satisfies $\langle u_g F, u_g F' \rangle = \alpha_g\big(\langle F,F' \rangle\big)$ for all $F,F' \in X_0/N$.
It therefore extends to an isometry on $X$, that we also denote by $u_g$, satisfying 
$$\langle u_g x,u_g y \rangle = \alpha_g(\langle x,y\rangle)$$ 
for all $x,y \in X$ (by continuity).
For $F$ as above and $a \in A$, we have
\begin{align*} u_g(F \cdot a) & = u_g\Big(\sum_{i=1}^n c(g_i) \cdot (a_i a)\Big) \\
& = \sum_{i=1}^n \big(c(gg_i) - c(g)\big) \cdot \alpha_g(a_i a) = (u_g F) \cdot \alpha_g(a) \,.
\end{align*}
Therefore, $u_g(x \cdot a) = u_g(x) \cdot \alpha_g(a)$ for all $x \in X$ and $a \in A$ (by continuity).

Consider now $g,h \in G$. For every $k \in G$ and $a \in A$, we have
\begin{align*}
(u_g u_h)(c(k) \cdot a) & = u_g\big((c(hk)-c(h)) \cdot \alpha_h(a) \big) \\
& = \big(c(ghk) - c(g)\big) \cdot \alpha_g(\alpha_h(a)) - \big(c(gh)-c(g)\big) \alpha_g(\alpha_h(a)) \\
& = (c(ghk) - c(gh)) \cdot \alpha_{gh}(a) = u_{gh}(c(k) \cdot a) \,. 
\end{align*}
Thus, by linearity, density and continuity, we get that $u_g u_h = u_{gh}$.
In particular, $$u_g u_{g^{-1}} = u_{g^{-1}} u_g = u_e = {\rm id}_X$$
(since $u_e(c(k) \cdot a) = (c(k) - c(e)) \cdot a = c(k) \cdot a$, as 
$c(e)=0$).
Hence each $u_g$ is invertible.

Altogether, we have shown that $u: g \mapsto u_g$ is an $\alpha$-equivariant action of $G$ on $X$.

Finally, by the definition of $u$, for all $g,h \in G$, we have $$u_g(c(h)) = \big(c(gh)-c(g)\big) \cdot \alpha_g(1_A) = c(gh)-c(g)\,.$$
So $c$ is a symmetric one-cocycle (w.r.t.~$u$). Since $\langle c(g),c(g) \rangle = \gamma(g,g) = \psi(g)$ for all $g \in G$,
we are done with the first two assertions of the theorem. 

If $\psi$ is assumed to be $Z(A)^+$-valued, then we see from (\ref{cgamma}) that $\langle c(g), c(h)\rangle$ belongs to $Z(A)$ for all $g,h\in G$, i.e., $c$ is central. 
\end{proof}

 Theorem \ref{cocyclerep} may probably be generalized to give a representation of any $\alpha$-negative definite function (see \cite{Del, Gui} for the classical case). However, for the time being
 we leave this as an open problem. 
 
\begin{remark}
The triple $(X,u,c)$ associated to $\psi$ in the previous theorem is unique in the following sense.
If $X'$ is another Hilbert $A$-module, equipped with an $\alpha$-equivariant action $u'$ of $G$ 
and a symmetric one-cocycle $c': G \to X'$ (w.r.t. $u'$) such that $\psi(g) = \langle c'(g),c'(g) \rangle'$ for all $g \in G$ and the $A$-submodule
of $X'$ generated by the $c'(g)$'s is dense in $X'$, then there exists a unitary operator $V$ from $X$ to $X'$ satisfying
$V u_g V^* = u'_g$ and $Vc(g) = c'(g)$
for all $g \in G$. 

To see this, the main observation is that we have
$\langle c(g),c(h) \rangle = \langle c'(g),c'(h) \rangle'$ 
for all $g,h \in G$; indeed,
\begin{align*}
2 \langle c(g),c(h) \rangle & = \psi(g) + \psi(h) - \alpha_g(\psi(g^{-1}h)) \\
& = \langle c'(g),c'(g) \rangle' + \langle c'(h),c'(h) \rangle' - \langle u'_g(c'(g^{-1}h)), u'_g(c'(g^{-1}h)) \rangle' \\
& = \langle c'(g),c'(g) \rangle' + \langle c'(h),c'(h) \rangle'  - \langle c'(h) - c'(g), c'(h) - c'(g) \rangle' \\
& = \langle c'(h),c'(g) \rangle' + \langle c'(g),c'(h) \rangle' \\
& = 2 \langle c'(g),c'(h) \rangle \,.  
\end{align*}
It is then easy to check that the map $V: X \to X'$ determined by 
$$V \Big(\sum_i c(g_i) \cdot a_i \Big) = \sum_i c'(g_i) \cdot a_i, \quad g_i \in G, a_i \in A$$ 
will do the job.
\end{remark}

\begin{remark}
It follows readily from Proposition \ref{cocyclenegdef} and Theorem \ref{cocyclerep} that the cone of $A^+$-valued normalized $\alpha$-negative definite
coincides with the set of functions of the form $g\mapsto \langle c(g), c(g)\rangle$ where $c$ ranges over all symmetric one-cocycles (with respect to $\alpha$-equivariant actions of $G$). Similarly, the subcone of $Z(A)^+$-valued normalized $\alpha$-negative definite
coincides with the set of functions of the form $g\mapsto \langle c(g), c(g)\rangle$  where $c$ ranges either over all symmetric and central one-cocycles,  or over all central one-cocycles (with respect to $\alpha$-equivariant actions of $G$).
 \end{remark}

\begin{remark}
Consider a function $\psi: G \to A^+$ given by $\psi(g) = \langle c(g),c(g) \rangle$ 
for some $\alpha$-equivariant action $u$ of $G$ on Hilbert $A$-module $X$ and a one-cocycle $c: G \to X$ (w.r.t.~$u$). 
Such a function will satisfy the first requirement, but not necessarily  the second, in the definition of $\alpha$-negative definiteness.
Instead of the second requirement, it will satisfy
$$\sum_{i,j=1}^n b_i^* \alpha_{g_i}\big(\psi(g_i^{-1}g_j)\big)b_j \leq 0 $$
for any $g_1,\ldots,g_n \in G$ and $b_1,\ldots,b_n \in Z(A)$ with $\sum_{i=1}^n b_i = 0$.
It might be worth to have a closer look at this class of functions in the future.
\end{remark}

A well known consequence of Schoenberg's theorem (see e.g.~\cite{BCR, BHV}) is that a function $\psi: G \to {\mathbb C}$ is negative definite
if and only if the function $\varphi_t:=e^{-t\psi}$ is positive definite for all $t > 0$. 
We now proceed to show that a version of this result continues to hold in our generalized setting, 
at least for central-valued functions. 
\begin{theorem}\label{schoenberg}
Let $\psi: G \to A$ and
consider the following two claims:
\begin{itemize}
\item[$(i)$] $\psi$ is $\alpha$-negative definite;
\item[$(ii)$] $e^{-t \psi}$ is $\alpha$-positive definite for all $t > 0$.
\end{itemize}
Then $(ii)$ implies $(i)$. \\
Moreover, suppose that $\psi$ is $Z(A)$-valued
 Then $(i)$ implies $(ii)$.
Thus, in this case,
the two claims above are equivalent.
\end{theorem}

\begin{proof}
Assume that $(ii)$ holds. Let $g\in G$. From \cite[Proposition 2.4]{AD1} we get that
$$\alpha_g\big(e^{-t\psi(g^{-1})}\big) = \big(e^{-t\psi(g)}\big)^*\,, $$
hence 
$$e^{-t\alpha_g(\psi(g^{-1}))} = e^{-t\psi(g)^*} $$
for all $t > 0$. Then \cite[Theorem VIII.1.2]{DS} gives that 
\begin{equation}\label{first}
\alpha_g(\psi(g^{-1})) = \psi(g)^*\,.
\end{equation}
Next, suppose $g_1, \ldots, g_n \in G$, $b_1, \ldots, b_n \in A$ with $\sum_{i=1}^n b_i = 0$ and let $\omega$ be a state on $A$.
By assumption, the scalar-valued function
$${\mathbb R} \ni t \mapsto \omega\Big(
\sum_{i,j=1}^n b_i^* \alpha_{g_i}\big(e^{-t\psi(g_i^{-1}g_j)}\big) b_j
\Big)$$ 
is non-negative for $t>0$ and vanishes at $t=0$. Thus
its right-derivative at $t=0$ must be non-negative, i.e.,
\begin{equation}\label{omeganeg}
- \omega\Big(\sum_{i,j=1}^n b_i^* \alpha_{g_i}\big(\psi(g_i^{-1}g_j) \big) b_j\Big) \geq 0\,.
\end{equation}
Now, using (\ref{first}) with $g=g_i^{-1}g_j$, we get
\begin{align*}\Big(\sum_{i,j=1}^n b_i^* \alpha_{g_i}\big(\psi(g_i^{-1}g_j) \big) b_j\Big)^* &= 
\sum_{i,j=1}^n b_j^* \alpha_{g_i}\big(\psi(g_i^{-1}g_j)^* \big) b_i \\
&= \sum_{i,j=1}^n b_j^* \alpha_{g_i}\big(\alpha_{g_i^{-1}g_j}\big(\psi(g_j^{-1}g_i) \big) \big) b_i \\
& = \sum_{i,j=1}^n b_j^* \alpha_{g_j}\big(\psi(g_j^{-1}g_i) \big) b_i\,, 
\end{align*}
which shows that $\sum_{i,j=1}^n b_i^* \alpha_{g_i}\big(\psi(g_i^{-1}g_j) \big) b_j$ is self-adjoint.
As (\ref{omeganeg}) holds for every state $\omega$ on $A$ we can therefore conclude that  
\[\sum_{i,j=1}^n b_i^* \alpha_{g_i}\big(e^{-t\psi(g_i^{-1}g_j)}\big) b_j \leq 0\,.\]
Thus we have shown that $\psi \in \nd$, that is, $(i)$ holds.

Suppose now that $\psi$
$\in \nd$ is $Z(A)$-valued.
In order to show that 
$(ii)$ holds in this case, it is enough    
to show that $e^{-\psi}$ is $\alpha$-positive definite, i.e., that
the $Z(A)$-valued matrix $\big[\alpha_{g_i}(e^{-\psi(g_i^{-1}g_j)})\big]$ is positive in $M_n(Z(A))$ for any given $g_1,\ldots,g_n \in G$. 
To this end, using the properties of the exponential function, we may write
$$\alpha_{g_i}\big(e^{-\psi(g_i^{-1}g_j)}\big) = e^{-\alpha_{g_i}(\psi(g_i^{-1}g_j))} 
= e^{\psi(g_i)^* + \psi(g_j) - \psi(e) -\alpha_{g_i}(\psi(g_i^{-1}g_j))} \ e^{-\psi(g_i)^* - \psi(g_j) + \psi(e)}  \  $$
for every $i, j$.
Setting $b_i = e^{\frac{1}{2}\psi(e)-\psi(g_i)} \in Z(A)$ for $i=1, \ldots, n$, we get that  
\[\Big[e^{-\psi(g_i)^* - \psi(g_j) + \psi(e)}\Big] = \big[\,b_i^*\,b_j\,\big] \]
is positive in $M_n(Z(A))$.
Therefore, using Lemma \ref{schur}, it is enough to show that 
\begin{equation}\label{exppos}
\Big[e^{\psi(g_i)^* + \psi(g_j) - \alpha_{g_i}(\psi(g_i^{-1}g_j))}\Big] \geq 0\,.
\end{equation}
 Now, Proposition \ref{gamma}  gives that 
the matrix $\Big[\psi(g_i)^* + \psi(g_j) - \alpha_{g_i}(\psi(g_i^{-1}g_j))\Big]$ is positive in $M_n(Z(A))$. Since Lemma \ref{schurexp} says that the Schur exponential of a $Z(A)$-valued positive matrix is still positive, 
we see that (\ref{exppos}) is satisfied. Hence we are done.

\end{proof}

\begin{corollary}\label{1-par}
Let $\psi$ be a 
normalized 
$Z(A)$-valued $\alpha$-negative definite function.
Then there exists a one-parameter semigroup $(M_t)_{t \geq 0}$ of unital completely positive maps on the full crossed product $C^*(A, G, \alpha)$
satisfying
$$M_t (F) =  e^{-t \psi} F $$
for all $t\geq0$ and all $F\in C_c(G,A)$. Moreover, if $\Lambda$ denotes the canonical homomorphism from $C^*(A, G, \alpha)$ onto the reduced crossed product $C_r^*(A, G, \alpha)$, then there also exists a one-parameter semigroup $(M^r_t)_{t \geq 0}$ of unital completely positive maps on $C_r^*(A, G, \alpha)$
satisfying
$$M^r_t \big(\Lambda(F)\big) =  \Lambda (e^{-t \psi} F) $$
for all $t\geq0$ and all $F\in C_c(G,A)$, i.e., $M^r_t \circ \Lambda = \Lambda \circ M_t$ for all $t\geq0$.
\end{corollary}

\begin{proof} Both statements follow by combining Theorem \ref{schoenberg} with \cite[Proposition 4.3]{BeCo6}. The second statement  can also be deduced from Theorem \ref{schoenberg} and \cite[Theorem 3.2]{DoRu}.
\end{proof}

\begin{remark} 
Let $(M_t)_{t \geq 0}$  be as described in Corollary \ref{1-par}.  Arguing as in \cite[Proposition 4.5]{Ren10} one obtains that the generator $-\Delta$ of this semigroup has the dense $*$-subalgebra $C_c(G,A)$ as its essential domain, and we have $\Delta F=\psi F$ for all $F\in C_c(G,A)$. 
(A similar remark is true for the semigroup $(M^r_t)_{t \geq 0}$.) Following Sauvageot \cite{Sau3} (see also \cite{Ren10}), one may then associate to $(M_t)_{t \geq 0}$ a Dirichlet form $\mathcal{L}$ on $C_c(G,A)$, which may be described in terms of a $C^*$-correspondence $E$ over $C^*(A,G,\alpha)$ and a derivation $\delta: C_c(G,A)\to E$. When $A=C(\Omega)$ is commutative, one may identify $C^*(A,G,\alpha)$ with the full $C^*$-algebra of the associated transformation groupoid $(G, \Omega)$. In this case, Renault gives in \cite[Theorem 4.6]{Ren10} a concrete description of the pair $(E, \delta)$. We believe it should be possible to obtain an analogous description also when $A$ is noncommutative. 
\end{remark}

We recall that a function $f: G \to A$ is said to {\it go to zero at infinity} if, for any $\epsilon >0$, there exists a finite subset $F \subset G$ such that
$\|f(g)\| < \epsilon$ for all $g \notin F$  (that is, $g \mapsto \|f(g)\| \in C_0(G)$). We denote by $C_0(G,A)$ the space of all such functions.

Next, assume that $G$ is countable. We  recall from \cite{DoRu} that 
$\alpha$ is said to have the {\it Haagerup property}  if there exists a sequence $(h_n)$ of 
$\alpha$-positive definite $Z(A)$-valued functions on $G$ such that $h_n(e)=1_A$ , $h_n \in C_0(G,A)$ for all $n \in \mathbb{N}$ 
and $\|h_n(g) - 1_A\| \to 0$ as $n \to \infty$ for all $g \in G$. (Note that Dong and Ruan's definition of $\alpha$-positive definiteness in \cite{DoRu} is slightly different than the one introduced in \cite{AD1}, but this is essentially a matter of convention and does not affect the definition of the Haagerup property for $\alpha$).  
It is easy to check that $\alpha$ has the Haagerup property  if 
the same property holds for a net $(h_\iota)_{\iota \in I}$ instead of a sequence.
It is a simple exercise to check that if $G$ has the Haagerup property, then $\alpha$ has the Haagerup property. 
On the other hand, if $\alpha$ has the Haagerup property and there exists an $\alpha$-invariant state $\omega$ on $A$, then $G$ has the Haagerup property (for if $(h_n)$ is sequence that works for $\alpha$, then  
$(\omega\circ h_n)$ will work for $G$). 

We will say that a function $\psi: G \to A^+$ is {\it spectrally proper} if the function \[\ell_\psi : g \mapsto  \inf {\rm sp}\big(\psi(g)\big)\] 
is proper as a function from $G$ to ${\mathbb R}^+$. Notice that this is a stronger property than requiring that the function $g \mapsto \|\psi(g)\|$ is proper in the usual sense.

The Haagerup property for a countable group $G$ may be characterized by the existence of a proper normalized negative definite function from $G$ into $\mathbb{R}^+$ (see \cite{CCJJV} and references therein). Analogously, we have:

\begin{theorem}\label{Haa-sp-prop}
Assume that $G$ is countable. Then  $\alpha$ has the Haagerup property if and only if and there exists there exists a spectrally proper $Z(A)^+$-valued normalized $\alpha$-negative definite function on $G$.
\end{theorem}

\begin{proof}
Assume first that  $\alpha$ has the Haagerup property and let $(h_n)$ be a sequence as in the definition. 
For each $n\in \mathbb{N}$ define $\varphi_n:G\to Z(A)^+$ by $\varphi_n(g) = h_n(g)^*h_n(g)$ for all $g\in G$. Then using Lemmas \ref{phistar} and \ref{posdefprod} we get that $(\varphi_n)$ is a sequence  in $C_0(G,A)$ of $Z(A)^+$-valued 
$\alpha$-positive definite  functions satisfying $\varphi_n(e) = 1_A$, and $\|\varphi_n(g) - 1_A\| \to 0$ as $n \to \infty$ for all $g \in G$.

Let now $(K_n)$ be an increasing and exhausting sequence $(K_n)$ of finite subsets of $G$.
Passing to a subsequence of $(\varphi_n)$ if necessary, we can assume that
$\|1_A - \varphi_n(g)\| \leq 1/2^n$ for all $n \in {\mathbb N}$ and $g \in K_n$. Since $\|\varphi_n(g)\| \leq \|\varphi_n(e)\| = 1$ (cf.~\cite[Proposition 2.4 ii)]{AD1}),
we get that $1_A - \varphi_n(g) \in Z(A)^+$ for all $n$ and $g$. Moreover, $(1-1/2^n)1_A \leq \varphi_n(g) \leq 1_A$ for all $n \in {\mathbb N}$ and $g \in K_n$.
Now, each function $1-\varphi_n$ is a $Z(A)^+$-valued normalized $\alpha$-negative definite function, cf.~Remark \ref{oneminus}.
Since $\sum_{j=1}^\infty \|1_A - \varphi_j(g)\| < +\infty$ for all $g \in G$, we can define $\psi:G\to Z(A)^+$ by $\psi(g)=\sum_{j=1}^\infty \big(1_A - \varphi_j(g)\big)$.
Using Lemma \ref{pointwise} we get that $\psi$ is a normalized $\alpha$-negative definite function. It remains to show that $\psi$ is spectrally proper.

For each $n \in {\mathbb N}$, using that $\varphi_n \in C_0(G,A)$, we can find a finite subset $F_n \subset G$ 
such that $\|\varphi_n(g)\| < 1/2$ for any $g \notin F_n$. Since $\varphi_n(g) \geq 0$, we have $\varphi_n(g) < \frac{1}{2}1_A$ for all $g \notin F_n$
and $K_n \subset F_n$ for each $n$.

Define $G_n = \bigcup_{j=1}^n F_j$, so $K_n \subset G_n$ and $(G_n)$ is an increasing and exhausting sequence of finite subsets of $G$.
Consider $g \notin G_n$. Then $\varphi_j(g) < \frac{1}{2}\, 1_A$ for $j=1,\ldots,n$, so 
$$\psi(g) = \sum_{j=1}^\infty (1_A - \varphi_j(g)) \geq \sum_{j=1}^n \frac{1}{2}1_A = \frac{n}{2} \,1_A \ . $$
Thus $\ell_\psi(g) \geq n/2$. It is now clear that $\ell_\psi$ is proper, i.e., $\psi$ is spectrally proper, as desired.

Conversely,  assume that there exists a spectrally proper $Z(A)^+$-valued normalized $\alpha$-negative definite function $\psi$ on $G$, and consider the net $(e^{-t \psi})_{t > 0}$. By Theorem \ref{schoenberg}, each $e^{-t \psi}$ is $\alpha$-positive definite and takes its values in
$Z(A)^+$. Clearly, $e^{-t \psi(e)} = 1_A$ for every $t>0$. Moreover, for $t >0$ and $g\in G$, we have \[\|e^{-t \psi(g)}\| = \sup \big\{ e^{-t\lambda} \ | \lambda \in {\rm sp}(\psi(g))\big\} = e^{-t \,\ell_\psi(g)}\,,\] which goes to $0$ as $g \to \infty$ for each $t>0$
since $\ell_\psi$ is proper. Thus $e^{-t \psi} \in C_0(G,A)$ for all $t>0$. 
Finally, it is clear that $\lim_{t \to 0} \|e^{-t \psi(g)} - 1_A\| = 0$ for all $t \in {\mathbb R}^+$. Hence we conclude that $\alpha$ has the Haagerup property.

\end{proof}

\begin{example}
Let us say that $\alpha$ is {\it centrally amenable}  if there exists a net $(h_i)$ of finitely supported
$\alpha$-positive definite $Z(A)$-valued functions on $G$ such that $h_i(e)=1_A$  for all $i$
and $\|h_i(g) - 1_A\| \to 0$ as $n \to \infty$ for all $g \in G$. Clearly, this is a stronger property than the Haagerup property for $\alpha$. We also note that if $\alpha$ is centrally amenable, then $\alpha$ is amenable in the sense of Anantharaman-Delaroche \cite{AD1} (and also as defined in \cite{BeCo6}). 

Now, assume that $\alpha$ is amenable as defined by Brown and Ozawa in their book \cite{BrOz}. Then $\alpha$ is centrally amenable. Indeed, if $(\xi_i)$ is a net satisfying the requirements of \cite[Definition 4.3.1]{BrOz}, then it is not difficult to see that the net $(h_i)$ in $C_c(G,A)$ defined by \[h_i(g) = \big\langle \xi_i, \widetilde{\alpha}_g(\xi_i)\big\rangle,\] 
where \[ \big\langle \xi, \eta\big\rangle = \sum_{s\in G} \xi(s)^*\eta(s)\quad \text{ and } \quad [\widetilde{\alpha}_g(\xi)](h) = \alpha_g\big(\xi(g^{-1}h)\big)\]
for $\xi, \eta \in C_c(G,A)$ and $g, h\in G$, satisfies all the conditions needed for showing that $\alpha$ is centrally amenable. The main point is that each $h_i$ is $\alpha$-positive definite, as follows from \cite[p.~300-301]{AD1}. Hence, if $G$ is countable,  we can conclude that $\alpha$ has the Haagerup property, and Theorem \ref{Haa-sp-prop} gives that there exists a spectrally proper $Z(A)^+$-valued normalized $\alpha$-negative definite function on $G$.
\end{example}

\begin{remark}
Recall (see e.g.~\cite{CCJJV, BHV}) that when $G$ is countable, then $G$ has property (T) if and only if every negative definite function from $G$ to $\mathbb{C}$ is bounded. 
One could therefore say that an action $\alpha$ has property (T) (resp.~has the central property (T)) if every $\alpha$-negative definite function (resp.~every center-valued $\alpha$-negative definite function) is bounded. Clearly $\alpha$ will have the central property (T) whenever it has property (T). Moreover,   $G$ will have property (T) whenever  $\alpha$ has the central property (T).

Indeed, assume $\alpha$ has the central property (T) and let $f:G\to \mathbb{C}$ be negative definite. Define $f_0:G\to \mathbb{C}$ by $f_0(g) = f(g) -f(e)$. Then $f_0$ is normalized and negative definite. Now let $\psi: G\to A$ be given by 
$\psi(g)= f_0(g)\, 1_A$.  Then $\psi$ is center-valued and normalized, and it follows from Remark \ref{scalarnegdef} that $\psi$ is $\alpha$-negative definite. Using the assumption, $\psi$ has to be bounded. So $f_0$ is bounded, and this clearly implies that $f$ is bounded too. Hence, $G$ has property (T).

Note that if $A$ has the strong property (T), as defined by Leung-Ng in \cite{LN},  and $G$ has property (T), then any $C^*$-crossed product of $A$ by $G$ also has the strong property (T) \cite[Theorem 4.6]{LN}. If one  assumes that  $\alpha$ has property (T), or the central property (T), it would be interesting to know if one can find some (weaker) conditions on $A$ ensuring that $C^*(A,G,\alpha)$ (or $C_r^*(A,G,\alpha)$) still has the strong property (T). 
\end{remark}

\noindent{\bf Acknowledgements.} 
Most of the present work was done during visits
made by E.B.~at the Sapienza University of Rome and by R.C.~at the 
University of Oslo in 2015 and 2016. 
Both authors would like to thank these institutions for their kind hospitality.

{\parindent=0pt Addresses of the authors:\\

\smallskip Erik B\'edos, Institute of Mathematics, University of
Oslo, \\
P.B. 1053 Blindern, N-0316 Oslo, Norway.\\ E-mail: bedos@math.uio.no \\

\smallskip \noindent
Roberto Conti, 
Universit\`{a} Sapienza di Roma, \\
 Dipartimento di Scienze di Base e Applicate per l'Ingegneria,\\
 via A. Scarpa 16, I-00166 Roma, Italy.
\\ E-mail: roberto.conti@sbai.uniroma1.it\par}

\end{document}